\documentclass[12pt]{paper}
\usepackage{amssymb,amsfonts,amsthm,amsmath}

\usepackage{latexsym,multicol,graphicx}

\usepackage{graphicx, verbatim}
\usepackage{xcolor}
\usepackage{verbatim}

\usepackage[a4paper, hmargin=3cm, vmargin={4cm, 4cm}]{geometry}

\usepackage{fancyhdr}

\usepackage{amssymb,amsfonts,amsthm,amsmath}
\usepackage{latexsym,multicol,graphicx}
\usepackage{float}

\newtheorem{theorem}{Theorem}[section]
\newtheorem{lemma}[theorem]{Lemma}
\newtheorem{definition}[theorem]{Definition}
\newtheorem{corollary}[theorem]{Corollary}

\newtheorem{rem}[theorem]{Remark}

\newcommand{\eps}{\varepsilon}
\newcommand{\vphi}{\varphi}
\usepackage{ulem}
\theoremstyle{plain}
\begin{document}
\title{Random  zero sets  for Fock type spaces }

\author{
Anna Kononova
}

\maketitle

{
\noindent{St.~Petersburg State University, Department of Mathematics and Mechanics,\\ 198504, Universitetsky pr., 28, Stary Peterhof, Russia}

\medskip
\noindent
{{\it Current address}:\\
School of Mathematical Sciences, Tel Aviv University, Tel Aviv, Israel}

\medskip
\noindent
{\tt anya.kononova@gmail.com}}

\medskip
\noindent
{ORCID: 0000-0002-4482-0540}

\abstract{Given a nondecreasing sequence $\Lambda=\{\lambda_n>0\}$ such that $\displaystyle\lim_{n\to\infty} \lambda_n=\infty,$ we consider the sequence $\mathcal N_\Lambda:=\left\{\lambda_ne^{i\theta_n},n\in\,\mathbb N\right\}$, where $\theta_n$ are independent random  variables   uniformly distributed  on $[0,2\pi].$ We discuss the conditions on the sequence $\Lambda$ under which  $\mathcal N_\Lambda$ is  a zero set  (a uniqness set) of a given weighted Fock space almost surely. The critical density of the sequence $\Lambda$ with  respect to the weight is found.

}
\medskip

\noindent{\bf Keywords:} entire function, Fock space, zero set, random rotations.  

\medskip

\section{Introduction}
Let $\varphi:\mathbb R_+\to \mathbb R_+$ be a positive increasing %$C^1$-smooth
function such that $$\lim_{t\to\infty}\varphi(t)=\infty.$$
An entire function $f$ is said to belong to the Fock type space $\mathcal F_\varphi^p$, where $0< p<\infty,$ if
\[\|f\|^p_{p,\varphi}:=
\int_{\mathbb C}\left|f(z)\right|^p
e^{-p\varphi(|z|)} {\rm d}m(z)<\infty
,
\]
where $m$ is  Lebesgue measure on $\mathbb C$.
The Fock spaces corresponding to the function  $\varphi(r)=r^2/2$ will be  called the {\it classical Fock spaces} and denoted by~$\mathcal F^p$.

A   sequence  $\mathcal N=\{z_k\in \mathbb C\}$ is called a {\it zero set} of the space $\mathcal F^p_\varphi$   if there exists a nonzero function  $f\in\mathcal F^p_\varphi$ whose zeroes are exactly the points from the sequence $\mathcal N$.
A sequence $\mathcal N$ is a {\it uniqueness set} of $\mathcal F^p_\varphi$, if the only function that vanishes at $\mathcal N$ is the zero function.

Let  $\Lambda=\{\lambda_n\}_{n\in\mathbb N}$ be a nondecreasing sequence of positive numbers such that $\lim \lambda_n =\infty.$
{ \begin{definition}
The  random sequence $\mathcal N_\Lambda$, obtained by rotating each element $\lambda_n$ of the sequence $ \Lambda$ by a random angle: $$\mathcal N_\Lambda:=\left\{z_n=\lambda_ne^{i\theta_n}: n\in\,\mathbb N=\{1,2,3,\ldots\}\right\},$$ where $\theta_n$ are independent random  variables   uniformly distributed  on $[0,2\pi],$   will be called the randomization of  $\Lambda$.
\end{definition}
}

In this note, we will be concerned with the following question: {\it  under which conditions on  $\Lambda$ the  random sequence $\mathcal N_\Lambda$  is a zero set of the Fock space $\mathcal F^p_\varphi$  almost surely?}

Speaking of a similar question for  other spaces of analytic functions, let us start with the  Hardy spaces $H^p(\mathbb D).$ Here  the situation is fully described by the Blaschke condition: the sequence $\mathcal N=\{z_n: |z_n|<1\}$ is a zero set of the space $H^p(\mathbb D)$ if and only if
$\sum (1-|z_n|)<\infty$, for any $p\in (0,\infty)$.
The situation with the Bergman spaces 
$A^p(\mathbb D)$ occurs to be quite different, as was originally shown in 1974 by Horowitz  \cite{Horowitz'74}.
The random approach to this question was firstly considered in 1990 by LeBlanc  \cite{LeBlanc}, who obtain a sufficient condition for the     set $\mathcal N_\Lambda$ with random angles  to be almost surely a  zero set of $A^2(\mathbb D)$, using the Blaschke-type product introduced earlier by Horowitz. Later this result was improved and extended by Bomash \cite{Bomash} and Horowitz \cite{Horowitz'94}.

Turning to the Fock spaces, we refer to the book \cite{Zhu} by Zhu for some general results and to the paper by Lyons and Zhai \cite{LyonsZhai}, where a discussion of the zero sets for Bergman and Fock spaces can be found.

{ While
it is a long standing problem to characterize deterministic zero sets of the Fock space, the probabilistic approach has some advantages.
}

Recently X.Fang and P.T.Tien in \cite{FT} found a sufficient condition on the sequence $\Lambda$ under which the sequence $\mathcal N_\Lambda$ is almost surely a zero set of the classical Fock space $\mathcal F^p_{\alpha t^2/2}$.
{In particular, they showed that if $\Lambda=\{\lambda_n\}$, where $\lambda_n\sim c\sqrt n$ as $n\to \infty$, then for all $\alpha> \frac{16}{(\sqrt {15} -3)c^2}$  and $p>0$ the randomized sequence $\mathcal N_\Lambda$ is a zero set of $ \mathcal F^p_{\alpha t^2/2}$ almost surely, while for $\alpha <1/c^2$ and $p>0$ the randomized sequence is  a uniqueness set for $\mathcal F^p_{\alpha t^2/2}$.  Since the bounds in these conditions do not match, the question of the critical value has been raised in the paper.
}

 It is  also worth  mentioning the paper \cite{CLP} by  Chistyakov, Lyubarskii and Pastur,  who dealt with another kind of randomization and showed that  a
random perturbation of the lattice $a(\mathbb Z+i\mathbb Z)=\{am + ian, m, n \in\mathbb Z\}$ is a.s. a zero set of the classical
Fock space if $a > \frac1{\sqrt \pi}$, and not a zero set if $a < \frac1{\sqrt \pi}$, under some natural conditions on the perturbation.

We will use the following notation and terminology.
\begin{definition}\label{n(t)}By $n(t)$ we denote the number of the elements of the sequence $\mathcal N$ in the open disk of radius $t$:  $$n(t)=\#\left\{z_n\in\mathcal N: |z_n|< t\right\}.$$
\end{definition}

\begin{definition}
Given an increasing differentiable function $\varphi:\mathbb R_+\to \mathbb R_+$ and a sequence $\mathcal N=\{z_n\in\mathbb C\}$ such that  $$\lim_{t\to\infty}\frac{n(t)}{t\varphi'(t)}=A,$$
we say that  $\mathcal N$ is 
\begin{itemize}
\item of {\bf critical density} with respect to the Fock space $\mathcal F^p_\varphi$, if $A=1$;
\item of {\bf subcritical density} with respect to the Fock space $\mathcal F^p_\varphi$, if $A<1;$
\item of {\bf supercritical density} with respect to the Fock space $\mathcal F^p_\varphi$, if $A>1.$
\end{itemize}
\end{definition}

The results of the present  paper are divided into three  parts. In Theorem \ref{Th1} we propose a (nonrandom, Jensen type) sufficient condition for the sequence $\mathcal N$  to be a uniqueness set of the Fock space $\mathcal F^p_{\varphi}$. In particular, it yields that any sequence $\mathcal N$ of supercritical density with respect to $\mathcal F^p_\varphi$ is not a zero set of this space. Some examples of the uniqueness sets with critical density are also given.

In  Theorem \ref{Th2} we show that, provided some  regularity of $\varphi$, for subcritical sequences $\Lambda$, the  randomization  $\mathcal N_\Lambda$ is a zero set of the Fock space $\mathcal F^p_\varphi$  almost surely.
In particular, it follows that for the sequence $\Lambda = \{\lambda_n\}$ such that $\displaystyle\lim_{n\to\infty}\frac{\lambda_n}{\sqrt{n}}>1$, the random sequence $\mathcal N_\Lambda$ is almost surely  a zero set for the classical Fock space $\mathcal F^p$, { which settles the question on the mismatching bounds  in \cite{FT}}.

The critical case is more subtle, and  we cannot  determine whether the random sequence $\mathcal N_\Lambda$ is a  zero set almost surely in terms of the asymptotic behavior of $n(t)$.
In Theorem \ref{Th3} we consider only the classical Fock space $\mathcal F^p$, and give  examples of  random a.s. zero sets with critical density with respect to the space $\mathcal F^p.$

In particular, it follows from our results that
\begin{itemize}
\item[1)] if $\Lambda\sim a \sqrt n$, then the sequence $\mathcal N_\Lambda$ is
\begin{itemize}
\item[(a)] a zero set of $\mathcal F^2$ for $a>1$ almost surely;
\item[(b)] a uniqueness set of $\mathcal F^2$ for $a<1$;
\item[(c)] a uniqueness set for $\lambda_n=\sqrt {n+\alpha}, \;\;\alpha\le 1/2$.
\end{itemize}
\item[2)] if $\Lambda$ is a  nondecreasing sequence consisting of the moduli of  all elements of the lattice $\mathbb Z+i\mathbb Z$:
 $$\Lambda= \sqrt{a\pi}\bigl|\mathbb Z+i\mathbb Z\bigr|=\sqrt{a\pi}\left\{0, 1, 1, 1, 1, \sqrt2,\sqrt2,\cdots\right\},$$ then the sequence $\mathcal N_\Lambda$ is
\begin{enumerate}
\item[(a)] a zero set of $\mathcal F^2$ for $a>1$ almost surely;
\item[(b)] a uniqueness set of $\mathcal F^2$ for $a<1$;
\item[(c)] {{\it Open question:} is $\mathcal N_\Lambda$ a uniqueness set in the critical case $a=1$?}  
\end{enumerate}
\end{itemize}

For the convenience of the reader, we add an Appendix at the end of the paper, which collects the results used in the proofs. 

\subsection* { Acknowledgements}
I am sincerely grateful to Alexander Borichev, who suggested the question studied in this paper and on several occasions helped with the proofs.
I would also like to express my gratitude to Evgeny Abakumov and Mikhail Sodin, who have read preliminary version of this work and made helpful comments. I'm also grateful to the referee, who carefully read the paper and made many remarks and suggestions, which helped to improve the exposition.

 The main results of this work were obtained while the author was supported by {\it Russian Science Foundation grant No. 20-61-46016}. On the final stage of preparation of this paper the author 
was supported by {\it
Israel Science Foundation Grant 1288/21}.

\section[Uniqueness set]{Sufficient  condition for the uniqueness set of the Fock space}\label{J}

In this section we give a (nonrandom) condition on any sequence $\Lambda$ sufficient for the sequence $\mathcal N$ such that $|z_n|=\lambda_n, n\in\mathbb N,$
to be a uniqueness set of $\mathcal F_{\varphi}^p.$
The following theorem  is an immediate consequence of Jensen formula and Jensen inequality. 

\begin{theorem}\label{Th1}
Let  $ g:\mathbb R_+\to (0,1]$ be  a  function such that $(\log g+\varphi)$ is  absolutely continuous, and
let $$\displaystyle
\int_1^\infty g^p(t){\rm d}t=\infty.$$
 A sequence $\displaystyle\mathcal N =\{z_n:\;|z_n|=\lambda_n,\;\forall n\in \mathbb N\}$ is a uniqueness  set  for the Fock space $\mathcal F^p_\varphi$, 
 provided that  there exists $M>0$ such that  
\begin{align}
 { n(t)}\ge t\cdot\left(\log g(t)+{\varphi(t)}\right)'-\frac 1p,\;\;\forall t\ge M.
\end{align}

\end{theorem}

If $\varphi$ is a differentiable function, taking $\displaystyle g(t)=t^{-1/p}$ we get
\begin{corollary}
 If there exists $M> 0$ such that
{$$n(t)+2/p\ge t\varphi'(t),\;\;\forall t\ge M,$$}
then {every} sequence $\displaystyle\mathcal N =\{z_n:\;|z_n|=\lambda_n,\;\forall n\in \mathbb N\}$ is a uniqueness  set  for the Fock space $\mathcal F^p_\varphi$.
\end{corollary}

\begin{corollary}
A sequence $\mathcal N$ of the supercritical density with respect to $\mathcal F^p_\varphi$  is a uniqueness  set  for the Fock space $\mathcal F^p_\varphi$.
\end{corollary}

\begin{proof}[Proof of Theorem \ref{Th1}]

We will prove it by contradiction. 
 Suppose that there is a function $f\in \mathcal F^p_\varphi$ such that the number of zeros $n(t)$ of this function satisfies the conditions of the theorem. Without loss of generality we can assume that $|f(0)|=1$. Indeed, if $f(0)=0 $ and $z=0$ is a zero of multiplicity $k$, we can consider a function $\displaystyle\widetilde f(z):=C\cdot f(z)\frac{(1-z)^k}{z^k}$ instead, where $\displaystyle C=\frac{k!}{f^{(k)}(0)}$ is the normalizing constant. The numbers of zeroes  of both functions   in the disk $\{z: |z|<t\}$  coincide for $t>1,$ and the conditions $f\in \mathcal F^p_\varphi$ and  $\widetilde f\in \mathcal F^p_\varphi$ are equivalent.

   Furthermore, by Jensen inequality for {$R>M$} we have
 \begin{multline*}
\log\int_{[0,2\pi]}|f(Re^{i\theta})|^p\frac{\rm d \theta}{2\pi}\ge p\int_{[0,2\pi]}\log|f(Re^{i\theta})|\frac{\rm d \theta}{2\pi}\\
=p\int_0^R \frac{n(t)}t{\rm d}t\ge {\int_M^R }\left (p\left(\log g(t)+\varphi(t)\right)'-\frac1t\right){\rm d}t\\= p\log g(R) +p\varphi(R) -\log R-{p(\log g(M)+\vphi(M))+\log M}.
\end{multline*}

Hence for some $C>0$ and for all {$R>M$}  
$$
\int_{[0,2\pi]}|f(Re^{i\theta})|^pe^{-p\varphi(R)}\frac{\rm d \theta}{2\pi}\ge C
\frac {g^p(R)}{R}.
$$

Consequently,  for $f\in\mathcal F^p_\varphi$ we get
\begin{multline*}\|f\|^p_{p,\varphi}=\int_{\mathbb C} |f(z)|^pe^{-p\varphi(|z|)}dm(z)\\\ge{\int_M^\infty} \int_{[0,2\pi]}
|f(Re^{i\theta})|^pe^{-p\varphi(R)}R{\rm d}\theta\;{\rm d}R\ge
C\int_{M}^\infty g^p(R){\rm d}R=\infty.
\end{multline*}
 The contradiction proves the theorem.
\end{proof}

{
Although the following corollary can  be derived from the Theorem \ref{Th1}, it is simpler to give a direct argument. We use the standard notation $\lfloor x\rfloor$ to denote the integer part of $x$.
\begin{corollary}
\label{nu} Every sequence $\mathcal N=\left\{z_n:\;|z_n|=\sqrt {n+\alpha},\;n\in\mathbb N
\right\},\; \alpha\le 1/2, $
is a uniqueness  set  for the  classical Fock  space $\mathcal F^2$.
\end{corollary}

\begin{proof}[Proof of Corollary \ref{nu}]

Let $f\in\mathcal F^2$ and $\lambda_n=|z_n|=\sqrt {n+\alpha}$. Assume, without loss of generality,  that $|f(0)|=1$.
Then 
$$
\log\int_{[0,2\pi]}|f(Re^{it})|^2\frac{\rm d t}{2\pi}\ge 2\int_{[0,2\pi]}\log|f(Re^{it})|\frac{\rm d t}{2\pi}=$$
$$=2\sum_{|z_n|\le R}\log\frac R{|z_n|}
\ge 2\lfloor R^2-\alpha\rfloor\log R-\log \Gamma \left(\lfloor R^2-\alpha\rfloor+\alpha+1\right)$$
$$= 2\lfloor R^2-\alpha\rfloor\log R-\left(\lfloor R^2-\alpha\rfloor+\alpha+1/2\right)
\log(\lfloor R^2-\alpha\rfloor+\alpha+1)+\lfloor R^2-\alpha\rfloor+O(1)$$
$$=-(2\alpha+1)\log R +R^2+O(1).$$

Hence, 
$$
\int_{[0,2\pi]}|f(Re^{it})|^2e^{-R^2}\frac{\rm d t}{2\pi}\ge 
C/R^{1+2\alpha},$$
and
$$\int_{\mathbb C} |f(z)|^2e^{-|z|^2}dm(z)\ge\int_1^\infty \int_{[0,2\pi]}
|f(Re^{it})|^2e^{-R^2}R{\rm d}t\;{\rm d}R\ge \int_1^{\infty} \frac{C}{R^{2\alpha}}{\rm d}R=+\infty.$$

\end{proof}
{{\it Open question:} How to describe the set of all $\alpha>0$ such that every sequence $\mathcal N=\left\{z_n:\;|z_n|=\sqrt {n+\alpha},\;n\in\mathbb N
\right\} $
is a uniqueness  set  for the  classical Fock  space $\mathcal F^2$?}

}

\section[Subcritical density]{Sufficient condition  for random zero set: \\ subcritical density}\label{L}

The  results of this section are based on the Levin-Pfluger theory of entire functions of so called completely regular growth (Appendix \ref{LP1}, \ref{LP2}), a detailed presentation of these results can be found in  Chapter II of  Levin's book \cite{Levin}. 
 We recall some definitions that will be used.

\begin{definition}
A function $\rho(t)$ is called a {\bf proximate order}, if it  satisfies the following conditions
\begin{align*}
0<\lim_{t\to\infty}\rho(t)=\rho <\infty;\\
\lim_{t\to\infty}t\rho'(t)\log t=0.
\end{align*}
\end{definition}

Given a sequence $\mathcal N\subset \mathbb C$ denote by
$n(t,\alpha, \beta)$  the number of the elements of  $\mathcal N$ in the sector $S(t,\alpha,\beta):=\{z\in\mathbb C: |z|<t, \arg z\in (\alpha, \beta)\}.$ 

\begin{definition} The sequence $\mathcal N$ is said to have an {\bf angular density} with respect to the proximate order  $\rho(r)$ if for any $\alpha$ and $\beta$ except for some countable set there exists a finite limit $$\lim_{r\to\infty}\frac{n(r,\alpha, \beta)}{r^{\rho(r)}}.$$ 
\end{definition}

The main result of this section is the following theorem.

\begin{theorem}\label{Th2}
 { Let  $\varphi(t)=t^{\rho(t)}$ where $\rho(t)$ is a proximate order. 
Suppose that a sequence 
$\Lambda=\left\{\lambda_k:\; \lambda_k\ge1\right\}$ is nondecreasing and
of  subcritical density with respect to the Fock space $\mathcal F^p_\varphi$.
Then the random sequence $\mathcal N_\Lambda$ is almost surely a zero set of the Fock space $\mathcal F_\varphi^p$.}
\end{theorem}

\begin{corollary}
Let $\Lambda:=\{\lambda_n\},$ where $\lambda_n\sim a\sqrt n, \;n\to\infty,$ with $a>1$. Then the sequence $\mathcal N_\Lambda$ is almost surely a zero set of the classical Fock space $\mathcal F^p$.
\end{corollary}\begin{corollary}
\label{lattice}
The sequence $\mathcal N_\Lambda$, where $\displaystyle\Lambda=\sqrt{a\pi}\left|\mathbb Z+i\mathbb Z\right|$ with $a>1$, is almost surely a zero set of the classical Fock space $\mathcal F^p$.
\end{corollary}
To show that $\Lambda$ from the Corollary \ref{lattice} satisfies the conditions of  Theorem \ref{Th2}, note that 
$$n(t)=\#\left\{(b,c)\in\mathbb Z^2: \sqrt{ a\pi}\left|b+ic\right|<t\right\}=\#\left\{(b,c)\in\mathbb Z^2: \left|b+ic\right|<\frac t{\sqrt{ a\pi}}\right\}=N\left(
\frac{t}{\sqrt{a\pi}}\right),$$
 here we use the standard notation from the Gauss circle problem and denote by $N(t)$  the number of integer lattice points inside
a circle of radius $t$. It is well known that
$$\lim_{t\to\infty}\frac{N(t)}{t^2}=\pi,$$
and so we have 
$$\lim_{t\to\infty}\frac{n(t)}{t^2}=\frac1a<1.$$

\begin{rem}\label{Rem}
Since $\varphi(t)=t^{\rho(t)},$ where $\rho(t)$ is a  proximate order, we have
$$\lim_{r\to\infty}\frac{r\varphi'(r)}{\varphi(r)}=
\lim_{r\to\infty}\left(\rho(r)+r\rho'(r)\log r \right)=\rho,
$$
and so  the subcriticality of the sequence $\Lambda$ is equivalent to the condition 
 \begin{align}\label{n(r)}
\lim_{r\to\infty}\frac{n(r)}{\varphi(r)}=a<\rho.
\end{align}
\end{rem}
Let us first prove some auxiliary results.

\subsection{}

In this section we show that a random sequence $\mathcal N_\Lambda$ in Theorem \ref{Th2} almost surely obeys some regularity properties.

The following lemma is to show that  $\mathcal N_\Lambda$  with probability one has an angular density of index $\rho(r)$, { which is equivalent, by  Remark \ref{Rem}, to the statement that with probability one for any $\alpha$ and $\beta$ the  limit 
$\displaystyle\lim_{r\to\infty}\frac {n(r,\alpha,\beta)}{n(r)}$
exists and is finite.
}

\begin{lemma}\label{ln}
With probability one the random sequence $\mathcal N_\Lambda$ obeys the following property:
$$\lim_{r\to\infty}\frac{n(r,2\pi x, 2\pi y)}{n(r)}=y-x,\;\;\;0\le x< y\le 1.$$
\end{lemma}
\begin{proof}
This lemma follows immediately from the Weyl-type criterion (Appendix, \ref{Holew}). Indeed, applying this criterion to the sequence $\displaystyle X_n=\frac{\theta_n}{2\pi}$
with $\displaystyle\phi_n(x)=\frac{e^{ix}-1}{ix},$  $n\in\mathbb N$, we get  for $x\in[0,1]$
$$ \lim_{r\to\infty} \frac{n(r,0, 2\pi x)}{n(r)}=
\lim_{N\to\infty} \frac{\#\{j:X_j< x\}}N =x\;\;\;{\mbox a.s.} $$
\end{proof}

\begin{lemma}\label{l1} Under the conditions   of  Theorem \ref{Th2}, if  $\rho\in\mathbb N$, then the series $\displaystyle \sum_k \frac 1{z_k^\rho}$ converges almost surely. {
Furthermore,  almost surely
\begin{equation}\label{tail}\displaystyle\lim_{r\to \infty} {r^{\rho-\rho(r)}}\cdot{\displaystyle\sum_{|z_k|> r} \frac 1{z_k^\rho}}=0.\end{equation}}
\end{lemma}

 \begin{proof}
 To justify  the convergence of the series of variances $ \displaystyle{\rm Var}\frac 1{ z_k^\rho}$, recall that by subcriticality of $\Lambda$, or, equivalently, by  (\ref{n(r)}) we have for some $C>0$
 \begin{align*}
  n(r)\le Cr^{3\rho/2}, \;\; r\ge 1.
 \end{align*}
 Since the sequence $\Lambda$ is  increasing, it follows that 
  \begin{align}\label{l^rho}
k\le n(\lambda_k+1)\le C\cdot(\lambda_k+1)^{3\rho/2},\;\;k\ge 1.
 \end{align}
  Finally,
  $$\displaystyle\sum {\rm Var}\frac 1{z_k^\rho}= \sum \lambda_k^{-2\rho}\le C_1 \sum k^{-4/3}<\infty.
  $$
    Now we can use the {Khinchine-Kolmogorov theorem (Appendix, \ref{3s})  for the random variables $\displaystyle Y_k={\rm Re}\frac 1{z_k^\rho}$ and $\displaystyle Y_k={\rm Im}\frac 1{z_k^\rho}$. 
 It follows that the series  $\displaystyle \sum_k \frac 1{ z_k^\rho}$ converges almost surely.}
 
 { Now we estimate the rate of convergence of this series. Since $\mathbb E\left(\displaystyle\sum_{|z_k|> r} \frac 1{z_k^\rho}\right)=0$ 
 and 
 $$\displaystyle {\rm Var}\left(\sum_{|z_k|> r}\frac 1{z_k^\rho}\right)= \sum_{\lambda_k> r} \lambda_k^{-2\rho}\le C_2 \int_r^\infty \frac{{\rm d}n(t)}{t^{2\rho}}\le C_3\left(\frac{n(r)}{r^{2\rho}}+\int_r^\infty\frac{t^{\rho(t)}}{t^{2\rho+1}}{\rm d}t\right)\le \frac{C_4}{\sqrt {r^\rho}},
  $$ 
  by Chebyshev's inequality we have
  
  $$\mathbb P\left(\left|\sum_{|z_k|> r}\frac 1{z_k^\rho}\right|\ge 
  r^{-\rho/ 8}\right)  \le C_5r^{-\rho/ 4}.
 $$
 Put $R_n=n^{5/\rho}$. Then
 $$\mathbb P\left(\left|\sum_{|z_k|> R_n}\frac 1{z_k^\rho}\right|\ge 
  n^{-5/8}\right)  \le C_5n^{-5/4}.
 $$ 
 Hence, by the Borel-Cantelli lemma, almost surely there exists { (a random)} $N\in\mathbb N$ such that for $n>N$
 \begin{equation}\label{Rate1}\left|\sum_{|z_k|> R_n}\frac 1{z_k^\rho}\right|< R_n^{-\rho/8}. \end{equation}

Let now $r\in(R_{n},R_{n+1}), \;\;n>N$.
Then
$$\left|\sum_{R_n<|z_k|\le r}\frac 1{z_k^\rho}\right|\le \sum_{R_n<|z_k|\le r}\frac 1{|z_k|^\rho}\le \frac{n(R_{n+1})-n(R_{n})}{R_n^\rho}.
$$
Recall that 
$$\lim_{r\to\infty}\frac{n(r)}{\varphi(r)}=a<\rho.
$$
Hence,
$$r^{\rho-\rho(r)}\left|\sum_{R_n<|z_k|\le r}\frac 1{z_k^\rho}\right|\le 
\frac{(a+o(1))\varphi(R_{n+1})-(a+o(1))\varphi(R_{n})}{\left(1+o(1)\right)\varphi(R_n)}.$$
By Remark \ref{Rem}
$$\log\frac{\vphi(R_{n+1})}{\vphi(R_{n})}= \int_{R_{n}}^{R_{n+1}}\frac{\vphi'(t){\rm d}t}{\vphi(t)}\le C_6\int_{R_{n}}^{R_{n+1}}\frac{{\rm d}t}{t}=C_6\log\frac{R_{n+1}}{R_{n}}=o(1),$$
and hence
$$\lim_{n\to 
\infty} \frac{\vphi(R_{n+1})}{ \vphi(R_{n})}=1.
$$
It follows that 
\begin{equation}
\label{Rate2}
r^{\rho-\rho(r)}\left|\sum_{R_n<|z_k|\le r}\frac 1{z_k^\rho}\right|=o(1),\;\;\;{r\to\infty}.
\end{equation}

Finally, summing up (\ref{Rate1}) and (\ref{Rate2}),
 we have almost surely
  \begin{align*}
  \lim_{r\to\infty}r^{\rho-\rho(r)}\left|\sum_{r<|z_k|}\frac 1{z_k^\rho}\right|&\le\lim_{r\to\infty}r^{\rho-\rho(r)}\left(\left|\sum_{R_n<|z_k|}\frac 1{z_k^\rho}\right|+\left|\sum_{R_n<|z_k|\le r}\frac 1{z_k^\rho}\right|\right)= 0.
  \end{align*}
 
 The lemma is proved.}

\end{proof}

{
}
\subsection{}
\begin{proof}[Proof of Theorem \ref{Th2}]

Given a random sequence $\mathcal N_\Lambda$ consider the corresponding Weierstrass canonical product
$$W(z):=\prod_{z_k\in \mathcal N_\Lambda} G\left(\frac{z}{z_k};\lfloor\rho\rfloor\right),$$
where  $G(w;d)$ are the elementary factors:
$$G(w;d)=(1-w)e^{w+\frac{w^2}{2}+\ldots+\frac{w^d}{d}}.
$$
It is well known (see \cite{Levin}, Chapter I), that the zero set of $W$ coincides with the set $\Lambda$.

By Lemma \ref{ln} the sequence $\mathcal N_\Lambda$ almost surely has angular density with index $\rho(r)$, and in case $\rho \notin \mathbb N$  we can apply the Levin-Pfluger theorem I (Appendix, \ref{LP1}). By this theorem, there exists a subset of a complex plane $E=\displaystyle\cup_j D_j$, where  $\{D_j\}_j$ is a set of disks  with zero linear density {(for the definition of this notion see (Appendix, \ref{Levin}))}, such that the following relation holds:
\begin{align}
\label{notN}
\lim_{\substack{r\to\infty,\\re^{it}\notin E}}\frac{\log|W(re^{it})|}{r^{\rho(r)}}=\frac a\rho<1.
\end{align}

In case $\rho\in\mathbb N$ we should take extra care about some kind of symmetry in the distribution of zeros of $W$, which is provided by   Lemma \ref{l1}.
 { By this lemma we know that there exists a (random, finite) number $S:=\sum z_k^{-\rho}$ and the series converges fast enough so that 
$$\delta:= \displaystyle\lim_{r\to \infty} {r^{\rho-\rho(r)}}\cdot{\displaystyle\sum_{|z_k|> r} \frac 1{z_k^\rho}}=0$$
almost surely. Now, consider the entire function
$$\widetilde W(z)=e^{S\cdot z^\rho}\prod G\left(\frac{z}{z_k};\rho\right).$$
By the Levin-Pfluger theorem II (Appendix, \ref{LP2}) there exists a subset of a complex plane $E=\displaystyle\cup_j D_j$, where  $\{D_j\}_j$ is a set of disks  with zero linear density (Appendix, \ref{Levin}), such that the following relation holds:
\begin{align}\label{inN}
\lim_{\substack{r\to\infty,\\re^{it}\notin E}}\frac{\log|\widetilde W(re^{it})|}{r^{\rho(r)}}=\frac a\rho<1.
\end{align}

}

Thus, by (\ref{notN}) and (\ref{inN}), by the maximum principle, with probability one  there exists a nonzero function,   $W\in\mathcal F^p_\varphi$ for $\rho\notin \mathbb N$, and  $\widetilde W\in\mathcal F^p_\varphi$ for $\rho\in \mathbb N$, with zeros exactly at the points of the random sequence $\mathcal N_\Lambda$ (by the construction).
That is,
$\mathcal N_\Lambda$ is almost surely a zero set of $\mathcal F^p_\varphi$.
\end{proof}

\section{Critical density}\label{H}

In this section we produce sequences $\Lambda$ of critical density with respect to the classical Fock space $\mathcal F^p$ such that the random sequence $\mathcal N_\Lambda$ with independent uniformly  distributed arguments is a zero set of $\mathcal F^p$.
\begin{rem}
Note, that  Corollary \ref{nu} provides us with an example of another nature --- a sequence of critical density that is not a zero set under any rotation of the arguments.
\end{rem} 
\begin{theorem}\label{Th3}
Let  $\mathcal N=\left\{z_n=\lambda_n e^{i\theta_n}\right\}$, where  { the sequence $(\lambda_n)$ is  strictly monotone, $\lambda_n\ge 1$ and \begin{equation}
    \label{lambdan}
\lambda_n^2=n+a\sqrt n\cdot\log^bn+o(1),\;\;\;n\to\infty,
\end{equation}} for some $a>0$, $b>3/2,$ and
$\theta_n$ are independent random variables uniformly distributed on $[0,2\pi].$
Then almost surely $\mathcal N$ is  a zero set for the classical Fock space $\mathcal F^p$ with any $p>0$. 
\end{theorem}

{
We will need the following technical lemma.
\begin{lemma}\label{lemman(t)}
Under the conditions of the theorem
\begin{enumerate} 
    \item $\lambda_{n+1}^2-\lambda_n^2\to 1,\;\;$ as $n\to\infty$,    \item $\displaystyle \lambda_{n+1}-\lambda_n\sim \frac 1{2\sqrt n},\;\;$ as $n\to\infty$,
    \item $n(t)=t^2-(a+o(1))\cdot t\cdot\log^bt^2,$ as $t\to\infty;$
    \end{enumerate}
\end{lemma}

\begin{proof}[Proof of Lemma \ref{lemman(t)}.]
 
The first two  claims follow immediately from (\ref{lambdan}).
    
To verify the third claim,    
    note that if $n=n(t)$, then  $\lambda_n<t\le \lambda_{n+1},$ and 
    $$t=\lambda_n+o(1)=\sqrt n+o(\sqrt n),$$ therefore
    \begin{align*}
        \lim_{t\to\infty}\frac{t^2-n(t)}{t\cdot\log^bt^2}=
        \lim_{n\to\infty}\frac{(\lambda_n+o(1))^2-n}{(\sqrt n+o(\sqrt n))\cdot\log^b(\sqrt n+o(\sqrt n))^2}\\=
        \lim_{n\to\infty}\frac{a\sqrt n\cdot\log^bn+o(\sqrt n)}{(\sqrt n+o(\sqrt n))\cdot\log^b(\sqrt n+o(\sqrt n))^2}=a.
    \end{align*}
    The lemma is proved.
\end{proof}

}
\begin{proof}[Proof of Theorem \ref{Th3}]
For $k\in\mathbb N$  denote 
$$H_k(z):=G\left(\frac{z}{z_k};1\right)=\left(1-\frac z{z_k}\right)e^{  z/{z_k}};$$

$$h_k(z):=\log|H_k(z)|=\frac12\log\left(1-2\frac {|z|}{\lambda_k}\cos(\beta-\theta_k)+\frac {|z|^2}{\lambda^2_k}\right)+\frac {|z|}{\lambda_k}\cos(\beta-\theta_k),$$
where $\beta=\arg z.$

Furthermore, for $|z|<|z_k|$
we have
$$h_k(z)={\rm Re}\left(\log\left(1-\frac z{z_k}\right)+\frac z{z_k}\right)=-{\rm Re}\sum_{j\ge 2}\frac {z^j}{jz^j_k}.
$$

We start with a calculation of the expectation of  $h_k(z)$, with $|
z|=R$ and $\arg z=\beta$:
 \begin{multline}\label{E}\mathbb E(h_k(z))= 
 \mathbb E\left(\log\left|1-\frac {R}{\lambda_k}e^{i(\beta-\theta_k)}\right|\right)+
 \mathbb E\left(\frac {R}{\lambda_k}\cos(\beta-\theta_k)\right)\\
 =\frac1{2\pi}\int\limits_{[-\pi,\pi]}\log\left|1-\frac {R}{\lambda_k}e^{i(\beta-\theta_k)}\right|{\rm d}\theta_k+\frac1{2\pi}\int\limits_{[-\pi,\pi]}\frac {R}{\lambda_k}\cos(\beta-\theta_k){\rm d}\theta_k\\=\begin{cases}
0, &R\le \lambda_k,\\
\log {\frac R{\lambda_k}},&R>\lambda_k.
\end{cases}
\end{multline}

\begin{lemma}\label{HoefSeries}
{ Under the conditions of Theorem \ref{Th3}  the function
$\displaystyle\prod_{k=1}^\infty H_k(z)$ is an entire function almost surely. 
}
\end{lemma}
 \begin{proof}
 {
 
 Consider the random variables $z_k^{-2}.$
 Since $\mathbb E(z_k^{-2})=0$ and ${\rm Var}(z_k^{-2})=\lambda_k^{-4}\sim k^{-2},$  by the  Khinchin-Kolmogorov theorem (Appendix, \ref{3s}) it follows that almost surely the series $\sum z_k^{-2}$ converges, so that  $\sigma:=\sum z_k^{-2}$ is a random number which is  almost surely finite.

 Consider an entire function determined by the Weierstrass canonical product $$K(z):=\prod_{z_k\in \mathcal N_\Lambda} G\left(\frac{z}{z_k};2\right).$$
Note, that for all
$k$ such that $|z_k|\ge 2|z|$ we have
 \begin{equation}\label{G} \log\left|G\left(\frac{z}{z_k};2\right)\right|={\rm Re}\left(\log G\left(\frac{z}{z_k};2\right)\right)\le \sum_{j=3}^\infty \left|\frac{z}{z_k}\right|^j\le 2\left|\frac{z}{z_k}\right|^3.
 \end{equation}
 
 Now,  $$\prod_{k=1}^\infty H_k(z) =\prod_{z_k\in \mathcal N_\Lambda} G\left(\frac{z}{z_k};1\right) = K(z)\exp\left(-z^2\sum_k{z_k^{-2}}\right)=K(z)\exp\left(-\sigma z^2\right),
 $$ hence it  is almost surely  an entire function.

The lemma is proved.}
\end{proof}

 Put
$$W(z):=\prod_{k=1}^\infty H_k(z).$$ By Lemma \ref{HoefSeries} it is an entire function with probability one.
We claim that  almost surely 
\begin{equation}
\label{Claim}
\limsup_{R\to\infty}\left(\max_{|z|=R}|W(z)|\right)\exp\left(-\frac{R^2}{2}\right)(R+1)^{3/p}\le 1,
\end{equation}
and, hence, $W\in \mathcal{F}^p$ for $1\le p<\infty$ with probability one.
 
 To prove relation (\ref{Claim}) we will estimate the probability of the large deviation of $W(z)$ from its expected value.
First, by (\ref{E}), the expected value of the random function  $\log|W(z)|$ at the point  $z, \;|z|=R$, is
\begin{align*} 
\mathbb E\left(\log|W(z)|\right)=&\mathbb E\left(\sum_{1}^{\infty}h_k(z)\right)=\sum_{\lambda_k<R} \log\frac R{\lambda_k}=\int_1^R \frac{{n(t)\rm d}t}{t}\\ =&\int_1^R \frac{{(t^2-{(a+o(1))t\log^{b}t^2)}\rm d}t}{t}=\frac{R^2}{2}-\int_1^R{(a+o(1))}\log^{b}t^2{\rm d}t\\
=&\frac{R^2}{2}-{(a+o(1))R\log^{b}R^2.
}\end{align*}

So, for $R$ large enough 
 $$\mathbb E\left(\log|W(z)|\right)\le \frac{R^2}{2}-\frac a2R\log^{b}R^2\le \frac{R^2}{2}-\frac a2R\log^{b}R.
$$

Set
$\displaystyle R_k=\frac{\lambda_k+\lambda_{k+1}}{2}$, $\displaystyle\eps_k~=~\frac{\lambda_{k+1}-\lambda_{k}}{4(\lambda_k+\lambda_{k+1})},\;k\ge 1.$

{ By Lemma \ref{lemman(t)} }\begin{equation}\label{lim_r_n}\displaystyle\lim_{k\to\infty}\frac{k}{R^2_k}=\lim_{k\to\infty}\frac{k}{\lambda^2_k}=16\cdot\lim_{k\to\infty}k\eps_k=1.
\end{equation}

Fix $k\ge 1$ and $z$ such that $|z|=R_k$.
The function $W$ is zero-free in the annulus $(1-\eps_k)R_k\le |z|\le(1+\eps_k) R_k$, and we may represent $W(z)$  as the product of two factors:
$$W_0(z):=\prod_{1\le s\le k} H_s(z);\;\;\;\;\;\;
W_{\infty}(z):=\prod_{s>k} H_s(z).$$

As above, $$\mathbb E\left(\log|W(z)|\right)=\mathbb E\left(\log|W_0(z)|\right)\le \frac{R^2}{2}-\frac a2R\log^{b}R.
$$ 
We will estimate the probability of the events $$\sup\limits_{|z|=R_k} \log|W(z)|\ge \frac{R_k^2}2-\frac{3}{p}\log R_k.$$
As the main tool we will use the 
 Bernstein-Hoeffding  concentration inequality (Appendix, \ref{Hoef}), applying it  to $W_0$ and $W_\infty$ separately.
As the first step we use  the concentration inequality for the value of $\log\left|W_0(z)\right|$ ($\log\left|W_\infty(z)\right|$) at an  individual point $z$ on the circle $|z|=R_k$, then we turn to the estimate of the maximum of  $\log\left|W_0(z)\right|$ ($\log\left|W_\infty(z)\right|$) over a finite number of points on this circle, and finally we prove that  the deviation of  $\log\left|W_0(z)\right|$ ($\log\left|W_\infty(z)\right|$)   from its expectation is small enough with probability close to one on the whole circle $|z|=R_k$.

\subsection{ Concentration inequality for $\log |W_0(z)|$.}

In this section, given $k\ge 1$, we estimate the probability of the large upward deviation of the quantity $\displaystyle\sup_{|z|=R_k}\log|W_0(z)|$ from the expectation of  $\log |W_0(z)|, |z|=R_k$.
 
 \subsubsection{One-point inequalities}

Put $$X_s:=h_s(z)-\mathbb E(h_s(z))=\begin{cases}h_s(z)-\log \frac{|z|}{\lambda_s},&1\le s\le k,\\
h_s(z),&s>k,
\end{cases},$$ 
where $|z|=R_k$.
First we estimate the size of  $X_s$ for $s\le k$:
\begin{align*}|X_s|=\left|h_s(z)-\log \frac{R_k}{\lambda_s}\right|=\left|\log \left|\frac {\lambda_s}{R_k}-e^{i(\beta-\theta_s)}\right|+\frac{R_k}{\lambda_s}\cos(\beta-\theta_s)|\right|\\\le\left|\log \left(1-\frac {\lambda_s}{R_k}\right)\right|+\frac{R_k}{\lambda_s}\le \log \frac1{\eps_k}+\frac{R_k}{\lambda_s}:=c_s.
\end{align*}
Next,
$$\sum_{s=1}^{k} c_s^2\le \sum_{s=1}^{k} 2(\log^2\eps_k +R_k^2/{\lambda_s^2}) = 2k\log^2\eps_k+2 R_k^2\int_1^{R_k}\frac{dn(t)}{t^2}{\le  C{R_k^2}(\log^2\eps_k+\log R_k)}.$$
By the Bernstein-Hoeffding inequality (Appendix, Section \ref{Hoef}) we have 
 \begin{align*}\mathbb P\left(\left|\log |W_0(z)|-\mathbb E\left(\log |W_0(z)|\right)\right|\ge \frac{a}{16}R_k\log^{b}R_k \right)\\
 \le \exp\left(-\frac{a^2 R_k^2\log^{2b}R_k}{C_1R_k^2(\log^2\eps_k+\log R_k)}\right)\\
\le \exp\left(-\frac{ a^2 \log^{2b}R_k}{C_1(\log^2\eps_k+\log R_k)}\right).
\end{align*}
Therefore, using relation (\ref{lim_r_n}), we conclude that
\begin{align*}\log\mathbb P\left(\left|\log |W_0(z)|-\mathbb E\left(\log |W_0(z)|\right)\right|\ge \frac{a}{16}R_k\log^{b}R_k \right)\\
 \lesssim -\frac{\log^{2b}R_k}{(\log^2\eps_k+\log R_k)}\lesssim -\log^{2(b-1)}R_k.
\end{align*}

\subsubsection{Multi-point inequalities}

Consider now $N_k$  equidistant points on the circle $|z|=R_k$ (the number $N_k$ is to be defined later):
$$\zeta_s:=R_k\exp\left(2\pi i \frac s{N_k}\right),\;\;s=0,\dots,N_k-1.$$
For each of the points $\zeta_s$ we have
$$\log\mathbb P\left\{\log|W_0(\zeta_k)|\ge \frac{R_k^2}2-\frac
{7a}{16}R_k\log^{b}R_k\right\}
\lesssim - \log^{2(b-1)}R_k.$$
Since
\begin{multline*}
    \mathbb P\left\{\sup_{0\le s< N_k}\log|W_0(\zeta_s)|\ge  \frac{R_k^2}2-\frac
{7a}{16}R_k\log^{b}R_k\right\}
\\\le
\sum_{0\le s< N_k}\mathbb P\left\{\log|W_0(\zeta_s)|\ge \frac{R_k^2}2-\frac
{7a}{16}R_k\log^{b}R_k\right\},
\end{multline*}
we obtain
$$\log\mathbb P\left\{\sup_{0\le s< N_k}\log|W_0(\zeta_s)|\ge  \frac{R_k^2}2-\frac
{7a}{16}R_k\log^{b}R_k\right\}
\lesssim- \log^{2(b-1)}R_k,$$
if $\log N_k=o(\log^{2(b-1)}R_k)$.
\subsubsection{Sup--inequalities}\label{Sup}

Now we are ready  to estimate the supremum of $\log |W_0(z)|$ on the whole circle $|z|=R_k$.
For $1\le s\le k$ we have
$$1<\lambda_s\le \lambda_k<R_k(1-\eps_k),
$$
and, hence,
\begin{align*}
\left|\frac{\partial}{\partial \beta}\log |W_0(R_ke^{i\beta})|\right|=\left|\sum_{s=1}^{k}\left(\frac{\frac{R_k}{\lambda_s}\sin(\beta-\theta_s)}{1+R_k^2/\lambda_s^2-2\frac{R_k}{\lambda_s}\cos(\beta-\theta_s)}-\frac{R_k}{\lambda_s}\sin(\beta-\theta_s)\right)\right|
\\\le
\sum_{s=1}^{k}\frac{R_k}{\lambda_s}
\left(\frac1{(\frac{R_k}{\lambda_s}-1)^2}+1\right)\le
k R_k
\left(\frac{(1-\eps_k)^2}{\eps_k^2}+1\right)\le  k R_k\frac 2{\eps_k^2}.
\end{align*}

Taking  $N_k=k^3$ and using (\ref{lim_r_n}) we obtain for $k>K,$ where $K$ is some deterministic constant,
$$\frac {2\pi}{N_k} \cdot\frac{2k R_k}{\eps_k^2}= {4\pi} \cdot\frac{ R_k}{(k\eps_k)^2}\le 15 R_k<\frac {a}{16}R_k\log^{b}R_k.
$$
Since
\begin{multline*}
\mathbb P\left\{\sup_{|z|=R_k}\log|W_0(z)|\ge  \frac{R_k^2}2-\frac
{6a}{16}R_k\log^{b}R_k\right\}\\
\le
\mathbb P\left\{\sup_{0\le s<N_k}\log|W_0(\zeta_s)|\ge\frac  {R_k^2}2-\frac
{6a}{16}R_k\log^{b}R_k-\frac {2\pi}{N_k} \cdot\frac{2k R_k}{\eps_k^2}\right\}\\
\le
\mathbb P\left\{\sup_{0\le s< N_k}\log|W_0(\zeta_s)|\ge  \frac{R_k^2}2-\frac
{7a}{16}R_k\log^{b}R_k\right\},
\end{multline*}
we get
\begin{equation}
\label{W0}
\log\mathbb P\left\{\sup_{|z|=R_k}\log|W_0(z)|\ge  \frac{R_k^2}2-\frac
{6a}{16}R_k\log^{b}R_k\right\}\lesssim-\log^{2(b-1)}R_k.
\end{equation}

\subsection{ Estimating  $\log |W_\infty(z)|$.}

We still keep the number $k\ge 1$ fixed. Our next step is to estimate  the probability that $\displaystyle\sup_{|z|=R_k}\log|W_\infty(z)|$ is "too large". 
 
\subsubsection{One-point inequalities}\label{One-point}

Consider a point $z$ on the circle $|z|=R_k$.
By Lemma \ref{HoefSeries} the series $\sum_{s=k}^\infty h_s(z)$ converges almost surely, let regroup it in the following way
$$\sum_{s>k} h_s(z)=\displaystyle\sum_{m=0}^\infty S_{m,k},\;\;\;{\rm where}\;\;\;S_{m,k}=\sum_{s=k2^m+1}^{k2^{m+1}} h_s(z),
$$
and estimate the probability of large deviation from zero of each $S_{m,k},$ $m=0,1,\ldots$.
Since 
\begin{multline*}
|h_s(z)|
=\left|-{\rm Re}\sum_{m\ge 2}\frac{z^m}{mz_s^m}\right|\le\frac{R_k^2}{2\lambda_s^2}\left(1+\sum_{m\ge 1}\frac{|z^m|}{m|z_s^m|}\right)
\le C \frac{R_k^2}{s}\log \frac1{\varepsilon_k} \le C\frac{R_k^2\log R_k}{s}=:c_{k,s},
\end{multline*}
where  $C$ is a deterministic constant, and
$$\sum_{s=k2^m+1}^{k2^{m+1}}c_{k,s}^2\le \frac{C_1R_k^4\log^2R_k}{k2^{m+1}},$${
using  the Bernstein-Hoeffding inequality (Appendix, \ref{Hoef}) we get
$$P_{m,k}:=\mathbb P \left(\left|S_{m,k}\right|\ge\frac{6}{\pi^2}\cdot\frac {aR_k\log^{b}R_k}{8\cdot(m+1)^2} \right)
\le
\exp\left(-C_2\frac{ 2^{m+1}}{(m+1)^4\displaystyle}\log^{2(b-1)}R_k\right),
$$
where $C_1, C_2$ are some deterministic constants.
 Hence,
 $$\mathbb P \left(\left|\sum_{s>k}h_s(z)\right|\ge\frac {a}{8} R_k\log^{b}R_k\right)\le 
\sum_{m=0}^\infty P_{k,m}\lesssim \exp\left(-C_3\log^{2(b-1)}R_k\right), $$
and finally 
$$\log\mathbb P \left(\log |W_\infty(z)|\ge\frac {a}8 R_k\log^{b}R_k\right)\lesssim -\log^{2(b-1)}R_k. $$
}

\subsubsection{Multi-point inequalities}

Consider $N_k$ equidistant points on the circle $|z|=R_k$:
$$\zeta_s:=R_k\exp\left(2\pi i \frac s{N_k}\right), \;\;\;0\le s<N_k.$$
For each of these points we have 
$$\log\mathbb P \left(\log |W_\infty(\zeta_s)|\ge\frac {a}8 R_k\log^{b}R_k\right)\lesssim -{\log^{2(b-1)}R_k}, $$
hence
$$\log\mathbb P \left(\sup_{0<s\le N_k}\left |W_\infty(\zeta_s)\right|\ge\frac {a}8 R_k\log^{b}R_k\right)\lesssim -\log^{2(b-1)}R_k$$
if $\log N_k=o(\log^{2(b-1)}R_k)$.

\subsubsection{Sup--inequalities}
Here we estimate the supremum of $\log |W_\infty(z)|$ on the whole circle $|z|=R_k$.   
We represent the angular derivative of $h_s$, $s>k$, in the  form
$$\frac{\partial}{\partial \varphi}h_s(z)=-{\rm Re}\sum_{l=2}^\infty \frac {iz^l}{z_s^l}= -{\rm Re} \left(\frac {iz^2}{z_s^2}+ \frac{iz^3}{z_s^3}\cdot\frac{z_s}{z_s-z}\right).
$$ 
 Therefore, (here $z=R_ke^{i\beta}$)
\begin{align*}
\left|\frac{\partial}{\partial \beta}\log |W_\infty(R_ke^{i\beta})|\right|\le&
\left|\sum_{s>k}\frac{z^2}{z_s^2}\right|+\sum_{s>k}\left|\frac{z^3}{z_s^3}\cdot\frac{z_s}{z_s-z}\right| \le
R_k^2
\left|\sum_{s>k}\frac{1}{z_s^2}\right|+R_k^3\sum_{s>k}\frac{1}{\lambda_s^2(\lambda_s-R_k)}\\
=&
R_k^2
\left|\sum_{s>k}\frac{1}{z_s^2}\right|+R_k^3\int_{(1+\eps_k) R_k}^\infty \frac{{\rm d}n(t)}{t^2(t-R_k)}\le 
R_k^2
\left|\sum_{s>k}\frac{1}{z_s^2}\right|+CR_k^2\log\left(1+\frac 1{\eps_k}\right).
\end{align*}

Now we estimate the probability of the event $\displaystyle\left|\sum_{s>k}\frac{1}{z_s^2}\right|>1.$ 
Using  the Bernstein-Hoeffding inequality (Appendix, \ref{Hoef}), we have, by the same way as  in the Section \ref{One-point},
\begin{multline*}\mathbb P \left(\left|\sum_{s>k}z_s^{-2}\right|\ge 1\right)\le 
\sum_{m=0}^\infty  
\mathbb P \left(\left|\sum_{s=k2^m+1}^{k2^{m+1}}z_s^{-2}\right|\ge\frac{6}{\pi^2(m+1)^2}\right)\\\le\sum_{m=0}^\infty  
\exp\left(-\frac{C_1k2^{m+1}}{(m+1)^4}\right)\le \exp(-C_2k),\end{multline*} where $C_1, C_2$ are some deterministic constants. 
Taking $N_k=k$ we obtain  for $k>K$, where $K$ is some deterministic constant, that  $$\log\mathbb P \left(\frac{2\pi}{N_k}\sup\limits_{|z|=R_k}\left|\frac{\partial}{\partial \beta}\log |W_\infty(z)|\right|\ge \frac {a}8 R_k\log^{b}R_k \right)\lesssim -k.
$$
Now, since 
\begin{multline*}
%\label{W_inf}
\mathbb P\left\{\sup_{|z|=R_k}\log|W_\infty(z)|\ge \frac {a}4 R_k\log^{b}R_k\right\}\\
\le
\mathbb P\left\{\sup_{0\le s<N_k}\log|W_\infty(\zeta_s)|+
\frac{2\pi}{N_k}\sup\limits_{|z|=R_k}\left|\frac{\partial}{\partial \beta}\log |W_\infty(z)|\right|
\ge \frac {a}4 R_k\log^{b}R_k\right\}\\
\le
\mathbb P\left\{\sup_{0\le s<N_k}\log|W_\infty(\zeta_s)|\ge \frac {a}8 R_k\log^{b}R_k\right\}\\+
\mathbb P\left\{\frac{2\pi}{N_k}
\sup\limits_{|z|=R_k}\left|\frac{\partial}{\partial \beta}\log |W_\infty(z)|\right|\ge \frac {a}8 R_k\log^{b}R_k\right\}
,\end{multline*}
for $k>K$ we get
\begin{equation}
\label{W_inf}
\mathbb \log\; {\mathbb P}\left\{\sup_{|z|=R_k}\log|W_\infty(z)|\ge \frac {a}4 R_k\log^{b}R_k\right\}\lesssim
- \log^{2(b-1)}R_k.
\end{equation}

Finally, by inequalities (\ref{W0}) and (\ref{W_inf}), there exists some determenistic constant $K_1$ such that  for $k>K_1$  we get
\begin{align}\label{-log}
\log\mathbb P\left\{\sup_{|z|=R_k}\log|W(z)|\ge \frac{R_k^2}2-\frac{a}8 R_k\log^{b}R_k\right\}
\lesssim
-\log^{2(b-1)}R_k.
\end{align}

\subsection{The end of the proof}
Given $C>0$, there exists $M\in\mathbb N$ such that for $k\ge M$ we have
$${\exp\left(-C{ \log^{2(b-1)}R_k}\right)\lesssim R_k^{-4}\lesssim k^{-2}}.$$
Hence, by the Borel-Cantelli lemma  
it follows from (\ref{-log}) that with probability one for all but a finite number of $k$ we have
\begin{equation}\label{*}\sup_{|z|=R_k}\log|W(z)|<  \frac{R_k^2}2-\frac {a}8 R_k \log^{b}R_k.
\end{equation}
From now on we assume that (\ref{*}) holds for $k\ge k_0.$

Now let $R_{k-1}\le R\le R_{k}, \;\;k\ge k_0.$ 
Then for every $z$  on the circle $|z|=R$ almost surely $$\sup_{|z|=R}\log|W(z)|\le \sup_{|z|=R_k}\log|W(z)|\le \frac{R_k^2}2-\frac a8 R_k \log^b R_{k}.$$
Recall that $R_k\sim\sqrt{k}$ and, { by  Lemma \ref{lemman(t)}},
$R_k-R_{k-1}\sim\displaystyle \frac{C}{\sqrt k}$ as $k\to\infty.
$
Hence, almost surely, for large $R
$  we have 
\begin{align*}\sup_{|z|=R}\log|W(z)|&<  \frac12\left(R_{k-1}+\frac{C_1}{\sqrt k}\right)^2-\frac a8R_k\log^b R_{k}\\&= \frac{R_{k-1}^2}2+C_1\frac{R_{k-1}}{\sqrt k}+\frac {C_1^2}{2k}-\frac a8R_k\log^b R_{k}\\&\le \frac{R^2}2-\frac3p\log (R+1).
\end{align*}
Thus, with probability one we have
$$\|W\|_{\mathcal F^p}=\int_{\mathbb C}|W(z)|^p\exp\left({-p\frac{|z|^2}{2}}\right){\rm d}m(z)\le C\int_{\mathbb C} \left(|z|+1\right)^{-3}{\rm d}m(z)<\infty.$$
\end{proof}

\section{Appendix}

Here we recall some notions and results which we have used in this note. 

%\subsection{Borel-Cantelli Lemma {\rm \cite[Theorem 8.3.4]{Dudley}}}\label{B-C}
%\subsection{Borel-Cantelli Lemma {\rm \cite[Theorem 4.3]{Bill}}}\label{B-C}

%{\it Let $\{A_n\}_{n\in\mathbb N}$ be a sequence   of events such that $\sum_n \mathbb P(A_n)<\infty$. Then 
%{
%$$\mathbb P\left(\lim \sup A_n\right )=0.$$
%}}

\subsection{Bernstein-Hoeffding concentration inequality {\rm \cite[Theorem 2.2.6]{Versh}}}\label{Hoef}
{\it  
Let $X_{1},X_{2},\dots ,X_{N}$ be independent random variables such that
${\displaystyle a_{i}\le X_{i}\leq b_{i}}$  for every $i$.
Then, for any $t>0$ we have
$$
\displaystyle \mathbb P \left(\sum_{n=1}^N \left(X_n-\mathbb E(X_{n})\right)\ge t\right)\le \exp \left(-{\frac {2t^{2}}{\sum _{i=1}^{N}(b_{i}-a_i)^{2}}}\right).
$$

%\subsection{Kolmogorov three-series theorem {\rm \cite[Theorem 9.7.3]{Dudley}}}\label{3s}
{
\subsection{Khinchine-Kolmogorov theorem on convergence of random series {\rm \cite[Theorem 22.6]{Bill}}}\label{3s}

{\it Let $(Y_k)_{k\in\mathbb N}$ be a sequence of independent real-valued random variables  with zero expectation. 
If 
$$\displaystyle\sum {\rm Var}\left(Y_k\right)<\infty,$$ then
the series $\displaystyle\sum Y_k$ convergence a.s.
%where
%$$Y_k=\begin{cases} Y_k, & |Y_k|\le 1,\\
%0, & |Y_k|> 1.
%\end{cases}$$ 
}

\subsection{Kolmogorov maximal inequality {\rm \cite[Theorem 22.4]{Bill}}}\label{KolmMax}

{\it Let $Y_1, Y_2,\ldots, Y_N$ be a sequence of independent real-valued random variables  with zero expectation and finite variances. For $\gamma>0$
$$\mathbb P\left(
\max\limits_{1\le k\le N} \left|\sum_{m=1}^{k} Y_m\right|\ge \gamma
\right)\\\le \frac 1{\gamma^2}\sum_{m=1}^N{\rm Var}(Y_m).
$$

}}}
%\subsection{Strong law of large numbers {\rm \cite[Theorem 1.3.1]{Versh}}}\label{LLN}
%{\it  
%Let $X_{1},X_{2},\ldots $ be independent identically distributed random variables such that $\mathbb E(X_1)=m$. Then almost surely
%$$\frac{X_1+\dots+X_N}{N}\to m,\;\; {\rm as} \;\;N\to\infty.
%$$
%}

\subsection{Weyl-type criterion of uniform distribution of random sequence {\rm \cite[Theorem 2]{Hol}}}\label{Holew}
{\it  
Let $X_{1},X_{2},\ldots $ be a sequence of  independent real-valued random variables with characteristic functions $\phi_1,\phi_2,\ldots$. Then this sequence is  uniformly distrubuted modulo 1 almost surely, that is 
$$\lim_{N\to\infty} \frac{\#\{k: X_k-\lfloor X_k\rfloor<x\}}N=x,\;\;\;\forall x\in [0,1)
$$
with probability one,
if and only if
for every $k\in\mathbb N$
$$\lim_{N\to\infty}\frac1{N}{\sum_{n=1}^N\phi_n(2\pi k )}=0.
$$
}

\subsection{Definition  of a set of disks with zero linear density {\rm \cite[Chapter II, \S 1]{Levin}}\label{Levin}}
{\it A set $D$ of disks $D_j$  in the complex plane is said to have  zero linear density, if 
$$\lim_{r\to\infty}\frac 1r\sum_{|c_j|<r} \lambda_j=0,
$$
where $\lambda_j$ is the radius  and $c_j$ is the center of  $D_j$.
}

\subsection{Levin-Pfluger theorem I {\rm \cite[Chapter II, \S 1, Theorem 1]{Levin}}}\label{LP1}
{\it Let a discrete set $\mathcal N\subset \mathbb C$ have an angular density of index $ \rho(r)$, where $\rho(r)$ is a proximate order with $\rho=\lim_{r\to\infty}\rho(r)\notin\mathbb N$. Let  $\Delta$ be a nondecreasing function such that for all but a countable set of angles
$$\Delta(\beta)-\Delta(\alpha)=\lim_{t\to\infty}\frac{n(t,\alpha,\beta)}{t^{\rho(t)}}.$$ 
Then for $z\in\mathbb C\setminus E$, where $E$ is a set of disks with zero linear density, the canonical product
$$W(z)=\prod_k G\left(\frac{z}{z_k};\lfloor\rho\rfloor\right)$$
satisfies the asymptotic relation
$$\lim_{r\to\infty}\frac{\log|W(re^{i\theta})|}{r^{\rho(r)}}=\frac{\pi}{\sin(\pi\rho)}\int_{\theta-2\pi}^{\theta}\cos\left(\rho(\theta-\psi-\pi)\right){\rm d}\Delta(\psi).$$}

\subsection{Levin-Pfluger theorem II {\rm \cite[Chapter II, \S 1, Theorem 2]{Levin}}}\label{LP2}
{\it Let a set $\mathcal N\subset \mathbb C$ have an angular density of index $\mathcal \rho(r)$, where $\rho(r)$ is a proximate order with $\rho=\lim_{r\to\infty}\rho(r)\in\mathbb N$. Let  $\Delta$ be a nondecreasing function such that for all but a countable set of angles
$$\Delta(\beta)-\Delta(\alpha)=\lim_{t\to\infty}\frac{n(t,\alpha,\beta)}{t^{\rho(t)}}.$$
 { Let the following limits exist and be finite
 $$S:=\sum z_n^{-\rho},
 $$
 $$\delta:=\lim_{r\to\infty}r^{\rho-\rho(r)}\sum_{|z_n|> r} \frac 1{ z_n^{\rho}}.$$
Then for $z\in\mathbb C\setminus E$, where $E$ is a set of disks with zero linear density, the entire function $$W(z)=e^{S\cdot z^{\rho}}\prod_k G\left(\frac{z}{z_k};\rho\right)$$
satisfies the following asymptotic relation
$$\lim_{r\to\infty}\frac{\log|W(re^{i\theta})|}{r^{\rho(r)}}=-\int_{\theta-2\pi}^{\theta}(\psi-\theta)\sin\left(\rho(\psi-\theta)\right){\rm d}\Delta(\psi) +\frac{|\delta|}{\rho}	\cos\left(\rho(\theta-\arg \delta)\right).$$}
}
%
%Suppose now that the sequence $\Lambda$ has a critical growth, that is
%$$\lim_{r\to\infty}\frac{n(r)}{r\varphi'(r)}=1.
%$$
%Let proceed as in the previous section and construct a function $W(z)=
%\prod_{k\ge 1}G\left(\frac z{z_k};\rho\right).$
%Consider a system of functions
%$$w_r(z):=W_N( zr).$$
%For $r\in (0,1)$ the zeros of function $w_r$ satisfy the condition 
%$$\lim_{t\to\infty}\frac{n(t)}{t\varphi'(t)}<1 
%,$$
%so by the Theorem 2
%$w_r\in \mathcal F^p_{\varphi},$
%moreover, by the same arguments 
%$$|w_r(z)e^{-\varphi(|t|)}|\le \frac{1}{|z|^3}.$$
%And by the monotone convergence theorem the limiting function
%$$\lim_{r\to 1-}w_r(z)e^{-\varphi(|z|)}=
%\lim_{r\to 1-}W_N(rz)e^{-\varphi(|z|)}=W_N(z)e^{-\varphi(|z|)}$$
%is also integrable with respect to the Lebesgue measure.
%

\medskip

\bigskip
\bigskip
\bigskip
\end{document}